\documentclass[11pt]{amsart}

\usepackage{amssymb}
\usepackage{eucal}
\usepackage{amsmath}
\usepackage{hyperref}
\usepackage{lmodern}
\usepackage{color}
\usepackage[T1]{fontenc}

\DeclareSymbolFont{SY}{U}{psy}{m}{n}
\DeclareMathSymbol{\emptyset}{\mathord}{SY}{'306}

\theoremstyle{plain}

\newtheorem{thm}{Theorem}[section]
\newtheorem{cor}[thm]{Corollary}
\newtheorem{lem}[thm]{Lemma}
\newtheorem{prop}[thm]{Proposition}
\theoremstyle{definition}
\newtheorem{defn}[thm]{Definition}
\newtheorem{rem}[thm]{Remark}

\numberwithin{equation}{section}

\def\C{{\mathbb C}}

\def\v{\varphi}
\def\l{\lambda}

\def\ra{\rightarrow}
\def\ov{\overline}
\def\lo{\longrightarrow}

\def\m{\mathcal}
\def\mb{\mathbb}
\def\a{\alpha}

\def\g{\gamma}

\def\wi{\widetilde}

\def\beq{\begin{eqnarray}}
\def\eeq{\end{eqnarray}}
\def\beqa{\begin{eqnarray*}}
\def\eeqa{\end{eqnarray*}}

\def\ov{\overline}
\def\bl{\boldsymbol}

\begin{document}
\title[$\Gamma_n$-isometries]{A Note on $\Gamma_n$-isometries}
\author[Biswas]{Shibananda Biswas}
\address[Biswas]{Theoretical Statistics and Mathematics Unit, Indian Statistical Institute, Kolkata 700108, India}
\email{shibananda@gmail.com}

\author[Shyam Roy]{Subrata Shyam Roy}
\address[Shyam Roy]{Indian Institute of Science Education and Research, Pincode 741252, Nadia, West Bengal, India}
\email{ssroy@iiserkol.ac.in}
\thanks{The work of S. Biswas was supported by Inspire Faculty Fellowship
funded by DST at Indian Statistical Institute, Kolkata.} \subjclass[2010]{47A13, 47A15, 47A25, 47A45} \keywords{Symmetrized 
polydisc, spectral set, Beurling-Lax-Halmos theorem, von Neumann inequality}
\begin{abstract}
In this note we characterize the distinguished boundary of the symmetrized polydisc and thereby develop a model theory
for $\Gamma_n$-isometries along the lines of \cite{AY}.  We further prove that for
invariant subspaces of $\Gamma_n$-isometries, similar to the case $n=2$ \cite{S}, Beurling-Lax-Halmos type representation holds.
\end{abstract}

\maketitle
\section{Introduction}
We denote by $\mb D$ and $\ov{\mb D}$ the open and closed unit discs in the complex plane $\mb C.$
Let $s_i,i\geq 0,$ be the \emph{elementary symmetric function} in $n$ variables of degree $i$, that is, $s_i$
is the sum of all products of $i$ distinct variables $z_i$ so that $s_0=1$ and
$$
s_i(z_1,\ldots ,z_n) = \sum_{1\leq k_1< k_2 <\ldots <k_i\leq n} z_{k_1} \cdots z_{k_i}.
$$
For $n \geq 1,$ let $\mathbf s:\C^n\lo\C^n$ be the function of symmetrization given by the formula
$$\mathbf s(z_1,\ldots,z_n)= \big (s_1(z_1,\ldots , z_n), \ldots , s_n(z_1,\ldots ,z_n) \big ).$$
The image $\Gamma_n:=\mathbf s(\ov{\mb D}^n)$ under the map $\mathbf s$ of the unit $n$-polydisc is known as
the {\it symmetrized $n$-disc}. The map $\mathbf s$ is a proper holomorphic map \cite{R}.

Following \cite{AY}, any commuting $n$-tuple of operators having $\Gamma_n$ as a spectral set will be called a
$\Gamma_n$-{\it contraction}. Many of the fundamental results in the theory of contractions have close parallels for
$\Gamma_2$-contractions as shown in \cite{AY}. 
In this paper we investigate properties of $\Gamma_n$-contarctions and we give a model for $\Gamma_n$-isometries. As an application,
we prove a Beurling-Lax-Halmos type theorem characterizing joint invariant subspaces of a pure $\Gamma_n$-isometry. We also
indicate how to construct a large class of examples of $\Gamma_n$-contractions.

Although a $\Gamma_2$-contraction can be obtained by symmetrizing any pair of commuting contractions \cite{AY}, it is no
longer true that the symmetrization of any $n$-tuple of commuting contractions will necessarily give rise to a
$\Gamma_n$-contraction, if $n>2$  (see Remark \ref{csym}). In fact, the symmetrization of an $n$-tuple of commuting contractions $(T_1,\ldots,T_n)$
is a $\Gamma_n$-contraction if and only if $(T_1,\ldots,T_n)$ satisfies the analogue of von Neumann's inequality
for all symmetric polynomials in $n$ variables (see Proposition \ref{vN}). However, it is shown in \cite[Examples 1.7 and 2.3]{AY} that not all $\Gamma_2$-contractions are obtained in this way.

We usually denote a typical point of $\Gamma_n$ by $(s_1,\ldots,s_n)$. We shall also use the notation $(S_1,\ldots,S_n)$
for an $n$-tuple of commuting operators associated in some way with $\Gamma_n$. In this paper an {\it operator} will always be a
bounded linear operator on a Hilbert space. The polynomial ring in $n$ variables over the field of complex numbers
is denoted by $\mb C[z_1,\ldots,z_n].$ Consider a commuting $n$-tuple $(S_1,\ldots,S_n)$ of operators. We say that
$\Gamma_n$ is a {\it spectral set} for  $(S_1,\ldots,S_n),$ or that  $(S_1,\ldots,S_n)$ is a $\Gamma_n$-\emph{contraction},
if, for every polynomial $p\in\mb C[z_1,\ldots,z_n],$
\beq\label{spec}
\Vert p(S_1,\ldots,S_n)\Vert\leq \sup_{\bl z\in\Gamma_n}\vert p(\bl z)\vert
=\Vert p\Vert_{\infty, \Gamma_n}.
\eeq
Furthermore, $\Gamma_n$ is said to be a \emph{complete spectral set} for $(S_1,\ldots,S_n)$, or $(S_1,\ldots,S_n)$
to be a \emph{complete $\Gamma_n$-contraction}, if, for every matricial polynomial  $p$ in $n$ variables,
\beqa
\Vert p(S_1,\ldots,S_n)\Vert\leq \sup_{\bl z\in\Gamma_n}\Vert p(\bl z)\Vert.
\eeqa
Here, if $S_1,\ldots, S_n$ act on a Hilbert space $\m H$ and the matricial polynomial $p$ is given by $p=[p_{ij}]$
of order $m\times \ell$, where each $p_{ij}$ is a scalar polynomial, then $p(S_1,\ldots,S_n)$ denotes the operator from $\m H^\ell$
to $\m H^m$ with block matrix $[p_{ij}(S_1,\ldots,S_n)].$ It is a deep result \cite[Theorem 1.5]{AY} that a $\Gamma_n$-contraction
is always a complete $\Gamma_n$-contraction and vice versa, for $n=2.$ It is not clear whether a similar result is
true for $n>2.$ This will be considered in a future work.

We denote the unit circle by $\mb T.$ The distinguished boundary of $\Gamma_n$, denoted by $b\Gamma_n,$ defined to be the Silov boundary of the
algebra of functions which are continuous on $\Gamma_n$ and analytic on the interior of $\Gamma_n$, is $\mathbf s(\mb T^n)$
\cite[Lemma 8]{EZ}. We shall use some spaces of vector-valued and operator-valued functions. We recall them following \cite{AY}.
Let $\m E$ be a separable Hilbert space. We denote by $\m L(\m E)$ the space of operators on $\m E,$ with the operator norm.
Let $H^2(\m E) $ denote the usual Hardy space of analytic $\m E$-valued functions on $\mb D$ and $L^2(\m E)$ the Hilbert
space of square integrable $\m E$-valued functions on $\mb T,$ with their natural inner products. Let $H^\infty\m L(\m E)$
denote the space of bounded analytic $\m L(\m E)$-valued functions on $\mb D $ and $L^\infty\m L(\m E)$ the space of bounded measurable
$\m L(\m E)$-valued functions on $\mb T,$ each with appropriate version of the supremum norm. For
$\v\in L^\infty \m L(\m E)$ we denote by $T_\v$ the Toeplitz operator with symbol $\v,$ given by
$$T_\v f=P_+(\v f),~~~~~ f\in H^2(\m E),$$
where $P_+:L^2(\m E)\lo H^2(\m E)$ is the orthogonal projection. In particular $T_z$ is the unilateral shift operator on
$H^2(\m E)$ (the identity function on $\mb T$ will be denoted by $z$) and $T_{\bar z}$ is the backward shift on $H^2(\m E).$

The symmetrization map $\mathbf s$ is a proper holomorphic map, $\ov{\mb D}^n=
\mathbf s^{-1}(\Gamma_n)=\mathbf s^{-1}(\mathbf s(\ov{\mb D}^n))$
and $\ov{\mb D}^n$ is polynomially convex. Therefore, $\mathbf s(\ov{\mb D}^n)=\Gamma_n$ is polynomially convex by
\cite[Theorem 1.6.24]{ELS}. Although we have defined
$\Gamma_n$-contractions by requiring that the inequality \eqref{spec} holds for all polynomials $p$ in $\mb C[z_1,\ldots,z_n],$
this is equivalent to the definition of $\Gamma_n$-contractions by requiring \eqref{spec} to hold for all functions $p$
analytic in a neighbourhood of $\Gamma_n$ due to polynomial convexity of $\Gamma_n$ as explained in \cite{AY}.
As discussed in \cite{AY}, the subtleties surrounding the various notions of joint spectrum and functional calculus for commuting tuples of operators
are not relevant to this paper, simply because of the polynomial convexity of $\Gamma_n.$

\section{$\Gamma_n$ and $\Gamma_n$-contractions}
Note that $(s_1,\ldots,s_n)\in \Gamma_n$ if and only if all the zeros of the polynomial $\sum_{i=0}^n (-1)^{n-i}s_{n-i}z^i$
lie in $\ov{\mb D}.$ This realization of points of $\Gamma_n$ will be used repeatedly. We state two theorems about location
of zeros of polynomials which will be useful in the sequel. For a polynomial $p\in\mb C[z],$ the derivative of $p$ with respect
to $z$ will be denoted by $p^\prime.$

\begin{thm}\label{GL}(Gauss-Lucas, \cite[page. 22]{M}) The zeros of the derivative of a polynomial $p$ lie in the convex hull of the
zeros of $p.$
\end{thm}
\noindent We recall a definition.
\begin{defn}
A polynomial $p\in\mb C[z]$ of degree $d$ is called \emph{self-inversive} if  $z^d\overline{p(\frac{1}{\bar z})}=\omega p(z)$
for some constant $\omega\in \mb C$ with $\vert \omega\vert=1.$
\end{defn}
\begin{thm} \label{Cohn}(Cohn, \cite[page, 206]{M})
 A necessary and sufficient condition for all the zeros of a polynomial $p$ to lie on the unit circle $\mb T$ is that
$p$ is self-inversive and all the zeros of $p^\prime$ lie in the closed unit disc $\ov{\mb D}.$
\end{thm}

We shall need characterizations of the distinghuished boundary of $\Gamma_n$.
\begin{thm}\label{bg}
 Let $s_i\in\mathbb C, i=1,\ldots,n.$ The following are equivalent:
 \begin{enumerate}
  \item[(i)] $(s_1,\ldots, s_{n})$ is in the distinguished boundary of $\Gamma_n$;
  \item[(ii)]$\vert s_n\vert=1, \bar s_n s_i=\bar s_{n-i}$ and $(\gamma_1s_1,\ldots,\gamma_{n-1}s_{n-1})\in\Gamma_{n-1},$
where $\gamma_i=\frac{n-i}{n}$ for $i=1,\ldots,n-1$;
  \item[(iii)] $(s_1,\ldots, s_{n})\in \Gamma_n$ and $\vert s_n\vert=1$.

 \end{enumerate}

\end{thm}
\begin{proof}
 Throughout this proof, we put $s_0=1.$ Let $(s_1,\ldots, s_{n})$ be in the distinguished boundary of $\Gamma_n$.
By definition, there are $\lambda_i\in \mathbb T$  such that
 \beq\label{elemen}
 s_i=\sum_{1\leq k_1<\ldots<k_i\leq n}\lambda_{k_1}\ldots \lambda_{k_i}\mbox{~for~} i=1,\ldots,n.
 \eeq
 This is equivalent to the fact that the polynomial $p,$ given
 by
\beq\label{p}
p(z)=\sum_{i=0}^n (-1)^{n-i}s_{n-i}z^i
\eeq
 has all its zeros on $\mathbb T$. Moreover, we clearly
 have $\vert s_n\vert=1, \bar s_n s_i=\bar s_{n-i}$  for $i=1,\ldots,n-1$. It follows from Theorem \ref{Cohn}
 that the polynomial
\beq\label{dp}
p^\prime(z)=\sum_{i=1}^{n}(-1)^{n-i}is_{n-i}z^{i-1}
\eeq
has all its roots in $\overline{\mathbb D},$
 which is equivalent to the fact that $(\gamma_1s_1,\ldots,\gamma_{n-1}s_{n-1})\in\Gamma_{n-1},$ where
 $\gamma_i=\frac{n-i}{n}$ for $i=1,\ldots,n-1.$ Therefore (i) implies (ii).

 Conversely, considering the polynomial in Equation \eqref{p}, we observe that
 \beqa
z^n\overline{p\bigg(\frac{1}{\bar z}\bigg)}=\sum_{i=0}^n(-1)^i\bar s_iz^i\mbox{~ and~}
(-1)^ns_n z^n\overline{p\bigg(\frac{1}{\bar z}\bigg)}=p(z)
\eeqa
 by the first part of (ii). Therefore $p$ is a self-inversive polynomial.
 Note that all the roots of $p^{\prime}$ lies in $\overline{\mathbb D}$ as $(\gamma_1s_1,\ldots,\gamma_{n-1}s_{n-1})\in\Gamma_{n-1},$  where
 $\gamma_i=\frac{n-i}{n}$ for $i=1,\ldots,n-1.$
 Thus, it follows from  Theorem \ref{Cohn}
 that $p$ has all  its roots on $\mathbb T.$ This is same as saying that $(s_1,\ldots,s_n)$ is in the distinguished
 boundary of $\Gamma_n.$

 Clearly, (i) implies (iii). To see the converse, we note that (iii) implies that there exist $\l_i\in\overline{\mb D},i=1,\ldots,n,$
 such that \eqref{elemen} holds and $\vert s_n\vert=\vert\l_1\vert\ldots\vert\l_n\vert=1.$ So $ \l_i\in \mb T$ for
all $i=1,\ldots,n.$
 It follows from the definition and \eqref{elemen} that $(s_1,\ldots,s_n)$ is in the distinguished boundary of $\Gamma_n.$
 \end{proof}
 \begin{rem}
 From the Lemma above, similar to the part (4) of \cite[Theorem 1.3]{BPS}, one can give expressions for $s_j$'s, $j=1,\ldots,n$. Clearly, $s_n =
 \l_1\ldots\l_n= e^{i\theta}$ for some $\theta$ in $\mathbb R.$ So $\l_n=e^{i\theta}\bar\l_1\ldots\bar\l_{n-1}.$ Since
 $$ s_j=\sum_{1\leq k_1< k_2 <\ldots <k_j\leq n} \l_{k_1} \cdots \l_{k_j}\mbox{~~with ~~}\vert \l_j\vert=1$$
 for $j=1,\ldots,n,$ we have
 \beqa
  s_j=\sum_{1\leq k_1< k_2 <\ldots <k_j\leq n} \l_{k_1} \cdots \l_{k_j}=\mu_j+\bar\mu_{j-1} \l_n
  \eeqa
where $$\mu_j=\sum_{1\leq k_1< k_2 <\ldots <k_j\leq n-1} \l_{k_1} \cdots \l_{k_j}.$$
Since $\l_n=e^{i\theta}\bar\l_1\ldots\bar\l_{n-1}$ we obtain
\beqa
s_j=\mu_j+\bar \mu_{j-1}e^{i\theta}\bar\l_1\ldots\bar\l_{n-1}=\mu_j+\mu_{j-1}\bar \mu_{n-1}e^{i\theta}
  = \mu_j+\mu_{n-j} e^{i\theta},
 \eeqa
 since $(\mu_1,\ldots,\mu_{n-1})\in b \Gamma_{n-1},\ \bar\mu_{n-1}\mu_{j-1}=\bar\mu_{n-j},\ j=1,\ldots, n-1.$

\end{rem}
 \begin{lem}\label{proj}
  If $(s_1,\ldots, s_n) \in \Gamma_n,$ then $(\gamma_1s_1,\ldots,\gamma_{n-1}s_{n-1})\in\Gamma_{n-1},$  where
  $\gamma_i=\frac{n-i}{n}$ for $i=1,\ldots,n-1.$
 \end{lem}
\begin{proof}
The hypothesis is equivalent to the fact that the polynomial $p(z)=\sum_{i=0}^n (-1)^{n-i}s_{n-i}z^i,$ has all its roots in the closed unit disc
$\overline{\mb D}.$ It follows from Theorem \ref{GL} that the polynomial $p^\prime(z)=\sum_{i=1}^{n}(-1)^{n-i}is_{n-i}z^{i-1}$ has all its roots
in the closed unit disc as well. Hence we have the desired conclusion.
\end{proof}

\begin{rem}
 Considering the map $\pi:\mb C^n\lo \mb C^{n-1}$ defined by $\pi(z_1,\ldots,z_n)=(\gamma_1z_1,\ldots,\gamma_{n-1}z_{n-1})$ the
above Lemma can be restated as $\pi(\Gamma_n)\subseteq \Gamma_{n-1}.$
\end{rem}

\begin{lem}\label{projection}
 If $(S_1,\ldots, S_n)$ is a $\Gamma_n$-contraction, then $(\g_1S_1,\ldots,\g_{n-1}S_{n-1})$ is a $\Gamma_{n-1}$-contraction,
   where $\g_i=\frac{n-i}{n}$ for $i=1,\ldots,n-1.$
\end{lem}
\begin{proof}
 For $p\in\mb C[z_1,\ldots z_{n-1}],$ we note that $p\circ \pi\in\mb C[z_1,\ldots,z_n]$ and by hypothesis
 \beqa
 \Vert p(\g_1S_1,\ldots,\g_{n-1}S_{n-1})\Vert&=&\Vert p\circ\pi(S_1,\ldots,S_n)\Vert\\
 &\leq& \Vert  p\circ\pi \Vert_{\infty,\Gamma_n}\\
 &= & \Vert  p\Vert_{\infty,\pi(\Gamma_n)}\leq \Vert  p\Vert_{\infty,\Gamma_{n-1}}.\\
 \eeqa This completes the proof.
\end{proof}


 \begin{lem}
If $(s_1,\ldots, s_n)\in\Gamma_n,$  then $(\a + s_1, \a s_1 + s_2,\ldots,\a s_{n-1} + s_n, \a s_n) \in \Gamma_{n+1}$
for all $\a$ in $\ov{\mathbb D}$.
 \end{lem}
 \begin{proof}
  If $(s_1,\ldots, s_n)\in\Gamma_n,$  then it follows from the definition of $\Gamma_n$ that there are $\l_k\in\ov{\mb D},
k=1,\ldots,n$ such that
$$s_i=\sum_{1\leq k_1< k_2 <\ldots <k_i\leq n} \l_{k_1} \cdots \l_{k_i},i=1,\ldots,n.$$
If $\a(:=\l_{n+1})\in\ov{\mb D},$ then $(\tilde s_1,\ldots, \tilde s_{n+1})\in\Gamma_{n+1},$ where
$$\tilde s_i=\sum_{1\leq k_1< k_2 <\ldots <k_i\leq n+1} \l_{k_1} \cdots \l_{k_i},i=1,\ldots,n+1.$$
Putting $s_0=1$ and $s_{n+1}=0,$ we note that $\tilde s_i=\a s_{i-1}+s_{i},i=1,\ldots, n+1.$
Therefore we have the desired conclusion.
\end{proof}
\begin{rem}
For any $\a\in \mb C,$  considering the one-to-one map $\pi_\a:\mb C^n\lo\mb C^{n+1}$ defined by
$$\pi_\a(z_1,\ldots, z_n)=(\a + z_1, \a z_1 + z_2,\ldots,\a z_{n-1} + z_n, \a z_n),$$
the above Lemma can be restated as $\pi_\a(\Gamma_n)\subseteq \Gamma_{n+1}$ for all $\a\in\ov{\mb D}.$
\end{rem}

\begin{lem}\label{embed}
 Let $(S_1,\ldots, S_n)$ be a $\Gamma_n$-contraction, then $(\a + S_1, \a S_1 + S_2,\ldots,\a S_{n-1} + S_n, \a S_n)$ is a $\Gamma_{n + 1}$-contraction for all $\a$ in $\ov{\mathbb D}$.
\end{lem}
\begin{proof}
 For $p\in \mb C[z_1,\ldots,z_{n+1}],$ we observe that $p\circ\pi_\a\in\mb C[z_1,\ldots,z_n],$ so by hypothesis we have
\beqa
&&\Vert p(\a + S_1, \a S_1 + S_2,\ldots,\a S_{n-1} + S_n, \a S_n)\Vert\\
&=&\Vert p\circ\pi_\a(S_1,\ldots,S_n)\Vert\\
&\leq&\Vert p\circ\pi_\a\Vert_{\infty,\Gamma_n}\\
&=&\Vert p\Vert_{\infty,\pi_\a(\Gamma_n)}\\
&\leq& \Vert p\Vert_{\infty,\Gamma_{n+1}}.\eeqa
Hence the proof.
\end{proof}
For an $n$-tuple $\mathbf T=(T_1,\ldots,T_n)$ of commuting operators on a Hilbert space $\m H,$ let $\mathbf s(\mathbf T)
:=(s_1(\mathbf T),\ldots,s_n(\mathbf T)),$ we call $\mathbf s(\mathbf T),$ the symmetrization of $\mathbf T.$ A polynomial
$p\in \mb C[z_1,\ldots,z_n]$ is called \emph{symmetric} if $p(\bl z_{\sigma}):=p(z_{\sigma(1)},\ldots,z_{\sigma(n)})=p(\bl z)$
for all $\bl z\in \mb C^n$ and $\sigma\in \Sigma_n,$ where $\Sigma_n$ is the symmetric group on $n$ symbols. If $p\in \mb C[z_1,\ldots,z_n]$
is symmetric, then there is a unique $q\in \mb C[z_1,\ldots,z_n]$ such that $p=q\circ \mathbf s,$ where $\mathbf s$ is the
symmetrization map \cite[Theorem 3.3.1]{E}.

\begin{rem}\label{csym}
 One may be tempted to conjecture that Lemma above is true when $\a$ is replaced by a contraction operator $T$ which
commutes with all the $S_i$'s, $i=1,\ldots,n$. However, this is no longer true. We take $n=2$ and give an example of
a $\Gamma_2$-contraction $(S_1,S_2)$ and a contraction $T$ such that $(T+S_1,TS_1+S_2,TS_2)$ is not a $\Gamma_3$-contraction.
Let $(A_1, A_2, A_3)$ be a tuple of commuting contractions as in Kaijser-Varopoulos \cite[Example 5.7]{paul}. We take
$S_1=A_1+A_2, S_2=A_1A_2$ and $T=A_3.$ Clearly,
$(S_1,S_2)=(A_1 + A_2, A_1A_2)$ is a $\Gamma_2$-contraction due to Ando's inequality. But
$(T+S_1,TS_1+S_2,TS_2)=\mathbf s(A_1, A_2 , A_3) $ is not a $\Gamma_3$-contraction due
to the failure of von Neumann's inequality for more than two commuting contractions. Consider the symmetric polynomial
 $$ p(z_1,z_2,z_3) = z_1^2 + z_2^2 + z_3^2 - 2z_1z_2 - 2z_2z_3 - 2z_3z_1$$
in \cite[Example 5.7]{paul}. Taking $q$ to be the polynomial in $3$-variables such that $q\circ \mathbf  s = p$, where
$\mathbf s$ is the symmetrization map, one observes that $\Gamma_3$ cannot be a spectral set for $\mathbf s(A_1, A_2 , A_3) .$
\end{rem}
\begin{prop}\label{vN}
The symmetrization of an $n$-tuple of commuting contractions $(T_1,\ldots,T_n)$
is a $\Gamma_n$-contraction if and only if $(T_1,\ldots,T_n)$ satisfies the analogue of von Neumann's inequality
for all symmetric polynomials in $\mb C[z_1,\ldots,z_n]$.
\end{prop}
\begin{proof}
 Let $\mathbf T=(T_1,\ldots,T_n)$ be a commuting $n$-tuple of contractions satisfying the analogue of von Neumann's
inequality for all symmetric polynomials $\mb C[z_1,\ldots,z_n]$.
 We show that $\mathbf s(\mathbf T)=(s_1(\mathbf T),\ldots, s_n(\mathbf T))$ is a $\Gamma_n$-contraction.
For a polynomial $p\in\mb C[z_1,\ldots,z_n],$ we note that
\beqa
\Vert p(s_1(\mathbf T),\ldots, s_n(\mathbf T))\Vert= \Vert p\circ \mathbf s( \mathbf T)\Vert
\leq \Vert p\circ \mathbf s\Vert_{\infty,\ov{\mb D}^n}
= \Vert p\Vert_{\infty,\mathbf s(\ov{\mb D}^n)},
\eeqa
since $p\circ\mathbf s$ is a  symmetric polynomial, the above inequality holds by hypothesis, and it shows that
$\Gamma_n=\mathbf s(\ov{\mb D}^n)$ is a spectral set for $(s_1(\mathbf T),\ldots,s_n(\mathbf T)).$

Conversely, let $\mathbf T=(T_1,\ldots,T_n)$ be  a tuple of commuting contractions such that $\mathbf s(\mathbf T)$ is a
$\Gamma_n$-contraction and  $q\in\mb C[z_1,\ldots,z_n]$ be symmetric. So there is a $p\in\mb C[z_1,\ldots,z_n]$ such that
$p\circ \mathbf s=q.$ By hypothesis, we have
\beqa
\Vert q(\mathbf T)\Vert=\Vert p( \mathbf s(\mathbf T))\Vert\leq \Vert p\Vert_{\infty,\Gamma_n}
=\Vert p\circ \mathbf s\Vert_{\infty,\ov{\mb D}^n}=\Vert q\Vert_{\infty,\ov{\mb D}^n}.
\eeqa
\end{proof}

Next, we give a straightforward generalization of the Lemma 3.2 in \cite{BPS}.
\begin{prop}\label{res}
Let $\mathbf T=(T_1,\ldots,T_n)$ be a commuting $n$-tuple of contractions on a Hilbert space $\m H$ satisfying the
analogue of von Neumann's inequality for all symmetric polynomials in $\mb C[z_1,\ldots,z_n].$ Let $\m M$ be an invariant
subspace for $s_i(\mathbf T), i=1,\ldots,n.$ Then $(s_1(\mathbf T)\vert_{\m M},\ldots,s_n(\mathbf T)\vert_{\m M})$ is a
$\Gamma_n$-contraction on the Hilbert space $\m M.$
\end{prop}
\begin{proof}
 For a polynomial $p\in\mb C[z_1,\ldots,z_n],$ we note that
\beqa
\Vert p(s_1(\mathbf T)\vert_{\m M},\ldots, s_n(\mathbf T)\vert_{\m M})\Vert
&=&\Vert p(s_1(\mathbf T),\ldots, s_n(\mathbf T))\vert_{\m M}\Vert\\
&\leq& \Vert p(s_1(\mathbf T),\ldots, s_n(\mathbf T))\Vert\\
&\leq& \Vert p\Vert_{\infty,\Gamma_n,}
\eeqa
 since $\mathbf s(\mathbf T)=(s_1(\mathbf T),\ldots,s_n(\mathbf T))$ is a $\Gamma_n$-contraction on $\m H,$ the last
inequality holds by the previous Proposition. Hence we have the desired conclusion.
 \end{proof}

\begin{rem}
${}$

\begin{enumerate}\label{ex}
\item [1.] It is clear from the proof of the Proposition above that if $(S_1,\ldots,S_n)$ is a $\Gamma_n$-contraction on a
Hilbert space $\m H$ and $\m M$ is a common invariant subspace  for $S_i,i=1,\ldots,n,$ then
$(S_1\vert_{\m M},\ldots,S_n\vert_{\m M})$ is a $\Gamma_n$-contraction on $\m M.$

\item[2.]As an immediate consequence of Proposition \ref{vN}, we observe that the symmetrizations of the classes of
$n$-tuples of commuting contractions discussed in \cite{GKVW} give rise to a large class of $\Gamma_n$-contractions.
\item[3.]It is shown in \cite{BPS} that applying the Lemma 3.2, a large class of $\Gamma_2$-contractions can be constructed.
In an analogous way, examples of $\Gamma_n$-contractions can be constructed for any integer $n> 2$ applying Proposition \ref{res}.
\item[4.] From Theorem 3.2 in \cite{AY}, it is clear that all $\Gamma_2$-contractions are obtained by applying the Lemma 3.2
in \cite{BPS} as described in the same paper . However, for $n>3$, it is not clear whether all $\Gamma_n$-contractions have similar realizations.

\end{enumerate}
\end{rem}

\section{$\Gamma_n$-unitary and $\Gamma_n$-isometries}
We start with the following obvious generalizations of definitions in \cite{AY}.
\begin{defn}
 Let $S_1,\ldots,S_n$ be  commuting operators  on a Hilbert space $\mathcal H$. We say that $(S_1,\ldots,S_n)$ is
\begin{enumerate}
 \item [(i)] a $\Gamma_n$-{\it unitary} if $S_i, i=1,\ldots,n$ are normal operators and the joint spectrum
$\sigma(S_1,\ldots,S_n)$ of $(S_1,\ldots,S_n)$ is contained in the distinguished boundary of $\Gamma_n.$
\item[(ii)] a $\Gamma_n$-{\it isometry} if there exists a Hilbert space $\mathcal K$ containing $\mathcal H$ and
a $\Gamma_n$-unitary $(\wi S_1,\ldots,\wi S_n)$ on $\m K$ such that $\m H$ is invariant for $\wi S_i, i=1,\ldots,n$
and $S_i={\wi S_i}\vert_{\m H}, i=1,\ldots,n.$
\item[(iii)] a $\Gamma_n$-{\it co-isometry} if $(S_1^*,\ldots,S_n^*)$ is a $\Gamma$-isometry.
\item[(iv)] a \emph{pure $\Gamma_n$-isometry} if $(S_1,\ldots,S_n)$ is a $\Gamma_n$-isometry and $S_n$ is a pure isometry.
\end{enumerate}

\end{defn}

The proof of the following theorem works along the lines of Agler and Young\cite{AY}.

 \begin{thm}\label{u}
  Let $S_i, \, i=1,\ldots,n,$ be commuting operators on a Hilbert space $\mathcal H$. The following are equivalent:
  \begin{enumerate}
   \item[(i)] $(S_1,\ldots,S_n)$ is a $\Gamma_n$-unitary;
   \item[(ii)] $S_n^*S_n=I=S_nS_n^*,\, S_n^*S_i=S_{n-i}^*$ and $(\g_1S_1,\ldots,\g_{n-1}S_{n-1})$ is a $\Gamma_{n-1}$-contraction,
   where $\g_i=\frac{n-i}{n}$ for $i=1,\ldots,n-1$;
   \item[(iii)] there exist commuting unitary operators $U_i$ for $i=1,\ldots,n$ on $\mathcal H$ such that
   $$S_i=\sum_{1\leq k_1<\ldots<k_i\leq n}U_{k_1}\ldots U_{k_i}\mbox{~for~} i=1,\ldots,n.$$
  \end{enumerate}
 \end{thm}

\begin{proof}Suppose (i) holds. Let $(S_1,\ldots, S_n)$ be a $\Gamma_n$-unitary. By the spectral theorem for commuting normal operators, there exists a spectral measure $M(\cdot)$ on $\sigma(S_1,\ldots,S_n)$ such that
 $$
 S_i = \int_{\sigma(S_1,\ldots,S_n)}s_i(\bl z) M(d\bl z),\, i=1,\ldots,n,
 $$
 where $s_1,...,s_n$ are the co-ordinate functions on $\C^n$. Let $\tau$ be a measurable right inverse of the restriction of $\bl s$ to $\mathbb T^n$, so that $\tau$ maps distinguished boundary of $\Gamma_n$ to $\mathbb T^n$. Let $\tau = (\tau_1,\ldots,\tau_n)$ and
 $$
 U_i = \int_{\sigma(S_1,\ldots,S_n)}\tau_i(\bl z) M(d\bl z),\, i=1,\ldots,n.
 $$
 Clearly $U_1,\ldots,U_n$ are commuting unitaries on $\mathcal H$ and
 \beqa
 \sum_{1\leq k_1<\ldots<k_i\leq n}U_{k_1}\ldots U_{k_i} &=& \sum_{1\leq k_1<\ldots<k_i\leq n}\int_{\sigma(S_1,\ldots,S_n)}\tau_{k_1}(\bl z)\ldots \tau_{k_i}(\bl z) M(dz)\\ &=& \int_{\sigma(S_1,\ldots,S_n)}s_i(\bl z) M(d\bl z)\\ &=& S_i,
 \eeqa
 for $i=1,\ldots,n$. Hence (i) implies (iii).

 Suppose (iii) holds. Then  $S_n^*S_n=I=S_nS_n^*$ and $ S_n^*S_i=S_{n-i}^*, i=1,\ldots,n$ follow immediately. Moreover, since
 $(S_1,\ldots,S_n)$ is a $\Gamma_n$-contraction, we have from Lemma \ref{projection} that $(\g_1S_1,\ldots,\g_{n-1}S_{n-1})$ is a $\Gamma_{n-1}$-contraction, where $\g_i=\frac{n-i}{n}$ for $i=1,\ldots,n-1.$ Hence (iii) implies (ii).

 Suppose (ii) holds. Since $S_n$ is normal, by Fuglede's theorem $S_{n-i}^*S_{n-i} = S_n^*S_iS_i^*S_n = S_iS_i^*$. Now as $S_i$'s commute, $S_{n-i}^*S_{n-i} = S_n^*S_iS_{n-i} = S_n^*S_{n-i}S_i =S_i^*S_i$ and we have each of $S_i, i= 1,\ldots,n$ is normal. So the unital $C^*$-algebra $C^*(S_1,\ldots, S_n)$ generated by $S_1,\ldots, S_n$
 is commutative and by Gelfand-Naimark's theorem is $*$-isometrically isomorphic to $C(\sigma(S_1,\ldots, S_n))$. Let ${\hat S}_1,\ldots,{\hat S}_n$ be the images of $S_1,\ldots, S_n$ under the Gelfand map. By definition, for an arbitrary point ${\bl z} = (s_1,,\ldots, s_n)$ in $\sigma(S_1,\ldots, S_n)$, ${\hat S}_i({\bl z}) = s_i$ for $i=1,\ldots,n$. By properties of the Gelfand map and hypothesis we have,
$$
 \overline{\hat S_n({\bl z})}\hat S_n({\bl z})= 1 = \hat S_n({\bl z}) \overline{\hat S_n({\bl z})}\mbox{~and~} \overline{\hat S_n({\bl z})}\hat S_i({\bl z}) = \overline{\hat S_{n -i}({\bl z})}\mbox{~for~} \bl z \mbox{~in~} \sigma(S_1,\ldots, S_n).
 $$
 Thus we obtain $\vert s_n\vert=1, \bar s_n s_i=\bar s_{n-i}$. Now $(\g_1S_1,\ldots,\g_{n-1}S_{n-1})$ is a $\Gamma_{n-1}$-contraction implies
 $$
 \Vert p(\g_1S_1,\ldots,\g_{n-1}S_{n-1})\Vert \leq \Vert p\Vert_{\infty,\Gamma_{n-1}}
 $$
 which is equivalent to $\Vert p\Vert_{\infty,\Gamma_{n-1}}^2 - p(\g_1S_1,\ldots,\g_{n-1}S_{n-1})^*p(\g_1S_1,\ldots,\g_{n-1}S_{n-1})$ is positive. Applying Gelfand transform we have
 $$
 \Vert p\Vert_{\infty,\Gamma_{n-1}}^2 - p(\g_1\hat S_1({\bl z}),\ldots,\g_{n-1}\hat S_{n-1}({\bl z}))^*p(\g_1\hat S_1({\bl z}),\ldots,\g_{n-1}\hat S_{n-1}({\bl z}))\geq 0
 $$ for $\bl z$ in $\sigma(S_1,\ldots, S_n)$. This shows $(\g_1s_1,\ldots,\g_{n-1}s_{n-1})$ is in the polynomially convex hull of $\Gamma_{n-1}$. Since $\Gamma_{n-1}$ is polynomially convex, $(\g_1s_1,\ldots,\g_{n-1}s_{n-1})$ is in $\Gamma_{n-1}$. Therefore by Theorem \ref{bg} $(s_1,\ldots,s_n)$ is in the distinguished boundary of $\Gamma_n$ and hence $\sigma(S_1,\ldots, S_n)\subset b\Gamma_n$. This proves (ii) implies (i).
 \end{proof}

It  is not a priori clear  whether unitarity of $S_n$ would imply that $(S_1,\ldots,S_n)$ to be $\Gamma_n$-unitary. However, the following is true.

 \begin{lem}
  Let ${\bf T} = (T_1,\ldots,T_n)$ be a tuple of commuting contractions on a Hilbert space $\mathcal H$. If $s_n{(\bf T)} {(= \prod_{i=1}^nT_i)}$ is a unitary, then ${\bl s}(\bf T)$ is a $\Gamma_n$-unitary.
 \end{lem}
\begin{proof}
 Note that each of the $T_i,\, i=1,\ldots,n$ is invertible as they are commuting with each other and their product is unitary. Since each $T_i$ is a contraction, we have $\|T_i^{-1}\| \geq 1$. First we show that $\|T_i^{-1}\| = 1$ for all $i=1,\ldots,n$. If not, then for some $k$ in $\{1,\ldots,n\}$ we have $\|T_k^{-1}\| > 1$. So, there exist a $y$ in $\mathcal H$ such that $\|T_k^{-1}y\| > \|y\|$. Let $(\prod_{i\neq k}^nT_i)y = x$. Thus
 $$
 \|s_n{(\bf T)}^{-1}\|\geq\frac{\|s_n{(\bf T)}^{-1}x\|}{\|x\|} = \frac{\|T_k^{-1}y\|}{ \|y\|}\cdot \frac{\|y\|}{\|(\prod_{i\neq k}^nT_i)y\|} > 1
 $$ is a contradiction as by hypothesis $s_n{(\bf T)}^{-1}$ is also a unitary. Next we show that each $T_i$ is isometry, that is, $T_i^*T_i = I,\, i=1,\ldots,n$. Suppose for some $\ell$ in $\{1,\ldots,n\}$, $T_\ell^*T_\ell \neq I$. There exists nonzero $y$ in $\mathcal H$ such that $(I - T_\ell^*T_\ell)y\neq 0$. Since $T_\ell$ is a contraction, $\|y\|^2 - \|T_\ell y\|^2 = \langle (I - T_\ell^*T_\ell)y, y\rangle = \|(I - T_\ell^*T_\ell)^{\frac{1}{2}}y\|^2 > 0$. But this shows, $\|T_\ell^{-1}\| > 1$ which is a contradiction. Thus $T_i^*T_i = I,\, i=1,\ldots,n$. Applying the same trick to ${\bf T}^* = (T_1^*,\ldots,T_n^*)$, we see that each $T_i$ is a unitary. Hence by part (iii) of the Theorem above,  ${\bl s}(\bf T)$ is a $\Gamma_n$-unitary.
\end{proof}

The following Lemma will be useful in characterizing pure $\Gamma_n$-isometry.

\begin{lem}\label{taut}
 Let $\Phi_1,\ldots,\Phi_n$ be functions in $L^\infty\mathcal L(\mathcal E)$ and $M_{\Phi_i},\,i=1\ldots,n$, denotes the corresponding multiplication operator on $L^2(\mathcal E)$. Then $(M_{\Phi_1},\ldots,M_{\Phi_n})$ is a $\Gamma_n$ contraction if and only if $(\Phi_1(z),\ldots,\Phi_n(z))$ is a $\Gamma_n$ contraction for all $z$ in $\mathbb T$.
\end{lem}
\begin{proof} Note that for $\|M_{\Psi}\| = \|\Psi\|_{\infty}$ for $\Psi\in L^\infty\mathcal L(\mathcal E)$. By definition,
$(M_{\Phi_1},\ldots,M_{\Phi_n})$ is a $\Gamma_n$ contraction if and only if
\beq\label{op}
\|p(M_{\Phi_1},\ldots,M_{\Phi_n})\| \leq \Vert p\Vert_{\infty,\Gamma_n,}
\eeq for all polynomials $p\in \mb C[z_1,\ldots,z_n]$. Since $p(M_{\Phi_1},\ldots,M_{\Phi_n}) = M_{p(\Phi_1,\ldots,\Phi_n)}$ and $\|M_{\Psi}\| = \|\Psi\|_{\infty}:= \sup_{z\in\mathbb T} \|\Psi(z)\|$ for $\Psi\in L^\infty\mathcal L(\mathcal E)$, we have \eqref{op} is true if and only if
$$
\|p(\Phi_1(z),\ldots,\Phi_n(z))\| \leq \Vert p\Vert_{\infty,\Gamma_n,}
$$ which completes the proof.
\end{proof}

\begin{rem}
 An interesting case in the Lemma above, is when $\dim \mathcal E = 1$. In this case, the $n$-tuple $(\Phi_1(z),\ldots,\Phi_n(z))$ is a $\Gamma_n$-contraction means that $(\Phi_1(z),\ldots,\Phi_n(z))$ is in $\Gamma_n$ which is true as $\Gamma_n$ is polynomially convex.
\end{rem}

 \begin{thm}\label{p.i}
  Let $S_i, \, i=1,\ldots,n,$ be commuting operators on a Hilbert space $\mathcal H$. Then $(S_1,\ldots, S_n)$ is a pure $\Gamma_n$-isometry if and only if there exist a separable Hilbert space $\mathcal E$ and a unitary operator $U: \mathcal H\ra H^2(\mathcal E)$ and functions $\Phi_1,\ldots,\Phi_{n-1}$ in $H^\infty \mathcal L(\mathcal E)$ and operators $A_i \in \mathcal L(\mathcal E)\,i=1,\ldots,n-1$ such that
  \begin{enumerate}
   \item[(i)] $S_i = U^*M_{\Phi_i}U,\, i=1,\ldots,n-1,\, S_n = U^*M_zU$;
   \item[(ii)] $(\g_1\Phi_1(z),\ldots,\g_{n-1}\Phi_{n-1}(z))$ is a $\Gamma_{n-1}$-contraction for all $z$ in $\mathbb T$, where $\g_i=\frac{n-i}{n}$ for $i=1,\ldots,n-1$;
   \item[(iii)] $\Phi_i(z) = A_i + A_{n-i}^*z$ for $1\leq i \leq n-1$;
   \item[(iv)] $[A_i, A_j] = 0$ and $[A_i, A_{n-j}^*] = [A_j, A_{n-i}^*]$ for $ 1\leq i,j \leq n-1,$ where $[P,Q]=PQ-QP$ for two operators $P,Q.$
  \end{enumerate}

 \end{thm}

  \begin{proof}
 Suppose $(S_1,\ldots, S_n)$ is a pure $\Gamma_n$-isometry. By definition, there exist a Hilbert space $\mathcal K$ and a $\Gamma_n$-unitary $(\tilde S_1,\ldots, \tilde S_n)$ such that $\mathcal H\subset\mathcal K$ is a common invariant subspace of $\tilde S_i$'s and $S_i = \tilde S_i|_{\mathcal H},\, i=1,\ldots,n$. From Theorem \ref{u}, it follows that ${\tilde S}_n^*{\tilde S}_n=I$ and
 ${\tilde S}_n^*{\tilde S}_i={\tilde S}_{n-i}^*, \, 1\leq i,j\leq n-1.$
 By compression of $\tilde S_i$ to $\mathcal H$, we have
 $$
 S_n^*S_n=I \mbox{~and~} S_n^*S_i=S_{n-i}^*, \, 1\leq i\leq n-1.
 $$
 Since $S_n$ is a pure isometry and $\mathcal H$ is separable, there exist a unitary operator $U: \mathcal H\ra H^2(\mathcal E)$, for some separable Hilbert space $\mathcal E$, such that $S_n = U^*M_zU$, where $M_z$ is the shift operator on $H^2(\mathcal E)$. Since $S_i$'s commute with $S_n$, there exist $\Phi_i\in H^\infty \mathcal L(\mathcal E)$ such that $S_i = U^*M_{\Phi_i}U$ for $1\leq i\leq n-1$. As $(S_1,\ldots, S_n)$ is a $\Gamma_n$-contraction, from Lemma \ref{projection} and Lemma \ref{taut} part (ii) follows. The relations $S_n^*S_i=S_{n-i}^*$ yield
 $$
 M_{\bar z} M_{\Phi_i} = M_{\Phi_{n-i}}^*,\, 1\leq i\leq n-1.
 $$
 Let $\Phi(z) = \sum_{n\geq 0}C_nz^n,\,\Psi(z) = \sum_{n\geq 0}D_nz^n$ be in $H^\infty \mathcal L(\mathcal E)$. Then $M_{\bar z} M_{\Phi} = M_{\Psi}^*$ implies
 $$
 C_0\bar z +C_1 + \sum_{n\geq 2}C_nz^{n-1} = D_0^*+D_1^*\bar z + \sum_{n\geq 2} D_n^*\bar z^n
 $$
 for all $z\in\mathbb T$ and by comparing coefficients, we get
 $$
 C_0 = D_1^* \mbox{~and~} C_1 =D_0^*.
$$ This gives that each $\Phi_i(z)$ is of the form $A_i + B_iz$ for some $A_i,B_i\in\mathcal L(\mathcal E)$, where $A_i = B_{n-i}^*$  for $1\leq i \leq n-1$, which is part (iii). Now since $S_i$'s commutes, we have  $M_{\Phi_i}M_{\Phi_j} = M_{\Phi_j}M_{\Phi_i}$ and consequently
$$
(A_i + A_{n-i}^*z)(A_j + A_{n-j}^*z) = (A_j + A_{n-j}^*z)(A_i + A_{n-i}^*z)
$$
for $1\leq i,j\leq n-1$ and for all $z\in\mathbb T$. Comparing the constant term and the coefficients of $z$, we get part (iv).

Conversely, suppose $(S_1,\ldots,S_n)$ be the $n$-tuple satisfying conditions (i) to (iv). Consider the $n$-tuple $(M_{\Phi_1},\ldots,M_{\Phi_{n-1}},M_z)$ of multiplication operators on $L^2(\mathcal E)$ with symbols $\Phi_1,\ldots,\Phi_{n-1},z$ respectively. From condition (iv), it follows that $M_{\Phi_i}$'s commutes with each other. Part of condition (iii) shows that $M_{\bar z} M_{\Phi_i} = M_{\Phi_{n-i}}^*$ by repeating calculations similar to above. Thus, it is easy to see from part (ii) of Theorem \ref{u}, that $(M_{\Phi_1},\ldots,M_{\Phi_{n-1}},M_z)$ is a $\Gamma_n$-unitary and so is $(U^*M_{\Phi_1}U,\ldots,U^*M_{\Phi_{n-1}}U,U^*M_zU)$. Now, $S_i$'s, $1\leq i\leq n$, are the restrictions to the common invariant subspace $H^2(\mathcal E)$ of $(M_{\Phi_1},\ldots,M_{\Phi_{n-1}},M_z)$ and hence $(S_1,\ldots,S_n)$ is a $\Gamma_n$-isometry. Since $S_n$ is a shift, $(S_1,\ldots,S_n)$ is a pure $\Gamma_n$-isometry.
 \end{proof}

 Next, we obtain characterization for $\Gamma_n$-isometry analogous to the Wold decomposition in terms of $\Gamma_n$-unitaries and pure $\Gamma_n$-isometries using Theorem \ref{u} and Theorem \ref{p.i}. We prove this along the way of Agler and Young \cite{AY} and will use the following Lemma \cite[Lemma 2.5]{AY}.
 \begin{lem}\label{ay}
  Let $U$ and $V$ be a unitary and a pure isometry on Hilbert space $\mathcal H_1, \mathcal H_2$ respectively, and let $T:\mathcal H_1\ra\mathcal H_2$ be an operator such that $TU = VT$. Then $T = 0$.
 \end{lem}

 \begin{thm}\label{wold}
  Let $S_i, \, i=1,\ldots,n,$ be commuting operators on a Hilbert space $\mathcal H$. The following are equivalent:
  \begin{enumerate}
   \item [(i)] $(S_1,\ldots, S_n)$ is a $\Gamma_n$-isometry;
   \item [(ii)] There exists a orthogonal decomposition $\mathcal H = \mathcal H_1\oplus\mathcal H_2$ into common reducing subspaces of $S_i$'s, $i=1,\ldots,n$ such that $(S_1|_{{\mathcal H}_1},\ldots, S_n|_{{\mathcal H}_1})$ is a $\Gamma_n$-unitary and $(S_1|_{{\mathcal H}_2},\ldots, S_n|_{{\mathcal H}_2})$ is a pure $\Gamma_n$-isometry;
   \item [(iii)] $S_n^*S_n=I,\, S_n^*S_i=S_{n-i}^*$ and $(\g_1S_1,\ldots,\g_{n-1}S_{n-1})$ is a $\Gamma_{n-1}$-contraction,
   where $\g_i=\frac{n-i}{n}$ for $i=1,\ldots,n-1.$
  \end{enumerate}
\end{thm}

\begin{proof} Suppose (i) holds. By definition, there exists $(\tilde S_1,\ldots,\tilde S_n)$ a $\Gamma_n$-unitary on $\mathcal K$ containing $\mathcal H$ such that $\mathcal H$ is a invariant subspace of $\tilde S_i$'s and $S_i$'s are restrictions of $\tilde S_i$'s to $\mathcal H$. From Theorem \ref{u}, it follows that $\tilde S_i$'s are satisfying the relations:
$
{\tilde S}_n^*{\tilde S}_n=I,\, {\tilde S}_n^*{\tilde S}_i={\tilde S}_{n-i}^*
$
 and $(\g_1{\tilde S}_1,\ldots,\g_{n-1}{\tilde S}_{n-1})$ is a $\Gamma_{n-1}$-contraction,
   where $\g_i=\frac{n-i}{n}$ for $i=1,\ldots,n-1$. Compressing to the common invariant subspace $\mathcal H$ and by part (1) of remark \ref{ex}, we obtain
 $$
 S_n^*S_n=I,\, S_n^*S_i=S_{n-i}^*
 $$
 and $(\g_1S_1,\ldots,\g_{n-1}S_{n-1})$ is a $\Gamma_{n-1}$-contraction, where $\g_i=\frac{n-i}{n}$ for $i=1,\ldots,n-1$. Thus (i) implies (iii).

 Suppose (iii) holds. By Wold decomposition, we may write $S_n = U\oplus V$ on $\mathcal H = \mathcal H_1\oplus\mathcal H_2$ where $\mathcal H_1,\mathcal H_2$ are reducing subspaces for $S_n$, $U$ is unitary and $V$ is pure isometry. Let us write
 $$
 S_i = \begin{bmatrix} S_{11}^{(i)} & S_{12}^{(i)}\\ S_{21}^{(i)} & S_{22}^{(i)} \end{bmatrix}
 $$
 with respect to this decomposition, where $S_{jk}^{(i)}$ is a bounded operator from $\mathcal H_k$ to $\mathcal H_j$. Since $S_nS_i = S_iS_n$, we have
 $$
 \begin{bmatrix} US_{11}^{(i)} & US_{12}^{(i)}\\ VS_{21}^{(i)} & VS_{22}^{(i)} \end{bmatrix} = \begin{bmatrix} S_{11}^{(i)}U & S_{12}^{(i)}V\\ S_{21}^{(i)}U & S_{22}^{(i)}V \end{bmatrix},\, i = 1,\ldots,n-1.
 $$
 Thus, $S_{21}^{(i)}U = VS_{21}^{(i)}$ and hence by Lemma \ref{ay}, $S_{21}^{(i)} = 0,\, i = 1,\ldots,n-1$. Now $S_n^*S_i=S_{n-i}^*$ gives
 $$
  \begin{bmatrix} U^*S_{11}^{(i)} & U^*S_{12}^{(i)}\\ 0 & V^*S_{22}^{(i)} \end{bmatrix} = \begin{bmatrix} S_{11}^{(n-i)*} & 0\\ S_{12}^{(n-i)*} & S_{22}^{(n-i)*} \end{bmatrix},\, i = 1,\ldots,n-1.
 $$
 It follows that $S_{12}^{(i)} = 0,\, i = 1,\ldots,n-1$. So $\mathcal H_1, \mathcal H_2$ are common reducing subspace for $S_1,\ldots,S_n$. From the matrix equation above, we have $U^*S_{11}^{(i)} = S_{11}^{(n-i)*},\, i = 1,\dots,n-1$. Thus by part (i) of the remark \ref{ex} and part (ii) of Theorem \ref{u}, it follows that $(S_{11}^{(1)},\ldots,S_{11}^{(n-1)},U)$ is a $\Gamma_n$-unitary.

 We now require to show that $(S_{22}^{(1)},\ldots,S_{22}^{(n-1)},V)$ is a pure $\Gamma_n$-isometry. Since V is a pure isometry $\mathcal H$ is separable, we can identify it with the shift operator $M_z$ on the space of vector valued functions $H^2(\mathcal E)$, for some separable Hilbert space $\mathcal E$. Since $S_{22}^{(i)}$'s commute with $V$, there exists $\Phi_i\in H^\infty \mathcal L(\mathcal E)$ such that $S_{22}^{(i)}$ can be identified with $M_{\Phi_i}$ for $1\leq i\leq n-1$. As $(S_1,\ldots, S_n)$ is a $\Gamma_n$-contraction, from Lemma \ref{projection} and Lemma \ref{taut}, it follows that $(\g_1S_{22}^{(1)},\ldots,\g_{n-1}S_{22}^{(n-1)})$ is a $\Gamma_{n-1}$-contraction,
   where $\g_i=\frac{n-i}{n}$ for $i=1,\ldots,n-1$. The relations $V^*S_{22}^{(i)} = S_{22}^{(n-i)*}$ yield
 $$
 M_{\bar z} M_{\Phi_i} = M_{\Phi_{n-i}}^*,\, 1\leq i\leq n-1.
 $$
 Calculations similar to the first part of the proof of Theorem \ref{p.i}, we get each $\Phi_i(z)$ is of the form $A_i + A_{n-i}^*z$ for some $A_i$'s in $\mathcal L(\mathcal E),\, i=1,\ldots,n-1$. Now since $S_{22}^{(i)}$'s commutes, we have
$$
[A_i, A_j] = 0 \mbox{~and~} [A_i, A_{n-j}^*] = [A_j, A_{n-i}^*] \mbox{~for~} 1\leq i,j \leq n-1.
$$
Hence, by Theorem \ref{p.i}, $(S_1|_{{\mathcal H}_2},\ldots, S_n|_{{\mathcal H}_2})$ is a pure $\Gamma_n$-isometry. Thus (iii) implies (ii).

It is easy to see that (ii) implies (i).
\end{proof}

  \begin{cor}
   Let $S_i, \, i=1,\ldots,n,$ be commuting operators on a Hilbert space $\mathcal H$. $(S_1,\ldots, S_n)$ is a $\Gamma_n$-co-isometry if and only if
   $$
   S_nS_n^*=I,\, S_nS_i^*=S_{n-i}
   $$ and $(\g_1S_1^*,\ldots,\g_{n-1}S_{n-1}^*)$ is a $\Gamma_{n-1}$-contraction,
   where $\g_i=\frac{n-i}{n}$ for $i=1,\ldots,n-1$.
  \end{cor}

\section{Characterization of invariants subspaces for $\Gamma_n$-isometries}
This section is devoted to characterize the joint invariant subspaces of pure $\Gamma_n$-isometries. Similar characterizations
for pure $\Gamma_2$-isometries appears in \cite{S}.

In the light of Theorem \ref{p.i}, we start with a characterization of pure $\Gamma_n$-isometries in terms of the parameters
associated with them. For simplicity of notation, let $(M_\mathbf \Phi,M_z)$ denote the $n$-tuple of multiplication operators
$(M_{\Phi_1},\ldots,M_{\Phi_{n-1}},M_z)$ on $H^2(\m E),$ where $\Phi_i\in H^\infty\m L(\m E),i=1,\ldots,n-1.$
Throughout this section, we will be using the canonical identification of $H^2(\mathcal E)$ with $H^2\otimes\mathcal E$ by the map $z^n\xi \mapsto z^n\otimes\xi$, where $n\in\mathbb N\cup\{0\}$ and $\xi\in\mathcal E$, whenever necessary.

\begin{thm}\label{equi}
Let $\Phi_i(z) = A_i + A_{n-i}^*z$ and $\tilde\Phi_i(z) = \tilde A_i +\tilde A_{n-i}^*z$ be in  $H^\infty \mathcal L(\mathcal E)$ for some  $A_i\in\mathcal L(\mathcal E)$ and $\tilde A_i\in\mathcal L(\mathcal F)$ respectively, $i=1,\ldots,n-1$. Then the $n$-tuple $(M_\mathbf \Phi,M_z)$ on $H^2(\mathcal E)$ is unitarily equivalent to the $n$-tuple $(M_{\tilde{\mathbf \Phi}},M_z)$ on $H^2(\mathcal F)$ if and only if the $(n-1)$-tuples $(A_1,\ldots,A_{n-1})$ and $(\tilde A_1,\ldots,\tilde A_{n-1})$ are unitarily equivalent.
\end{thm}

\begin{proof}
Suppose the $n$-tuple $(M_{{\mathbf \Phi}},M_z)$ on $H^2(\mathcal E)$ is unitarily equivalent to the $n$-tuple $(M_{\tilde{\mathbf \Phi}},M_z)$ on $H^2(\mathcal F)$. We can identify the map $\Phi_i$ (similarly $\tilde\Phi_i$) by $I_{H^2}\otimes A_i + M_z\otimes A_{n-i}^*$. So, there exist a unitary $U : H^2\otimes\mathcal E\ra H^2\otimes\mathcal F$ such that
\beq\label{ue}
U(I_{H^2}\otimes A_i + M_z\otimes A_{n-i}^*)U^* = I_{H^2}\otimes \tilde A_i + M_z\otimes \tilde A_{n-i}^*,\, i=1,\ldots,n-1,
\eeq and
\beq\label{ue1}
U(M_z\otimes I_{\mathcal E})U^* = M_z\otimes I_{\mathcal F}.
\eeq
 From equation \eqref{ue1}, it follows that there exists a unitary $\tilde U :\mathcal E\ra\mathcal F$ such that $U = I_{H^2}\otimes \tilde U$. Consequently, the equation \eqref{ue} can be written as
 $$
 \tilde U A_i \tilde U^* + \tilde U A_{n-i}^*z\tilde U^* z = \tilde A_i +\tilde A_{n-i}^*z,\,i=1,\ldots,n-1,
 $$ for all $z\in\mathbb T$. Hence comparing the coefficients, we obtain $\tilde U A_i \tilde U^* = \tilde A_i,\,i=1,\ldots,n-1$ which completes the proof in forward direction.

Conversely, suppose there exist a unitary $\tilde U :\mathcal E\ra\mathcal F$ that intertwines $A_i$ and $\tilde A_i$, that is, $\tilde U A_i \tilde U^* = \tilde A_i$ for each $i = 1,\ldots,n-1$. Let $U: H^2\otimes\mathcal E\ra H^2\otimes\mathcal F$ be the map defined by $U = I_{H^2}\otimes \tilde U$. Clearly, $U$ is a unitary and from the computations similar to above, it is easy to see that $U$ intertwines $M_{\Phi_i}$ with $M_{\tilde \Phi_i}$ for each $i=,1\ldots,n-1$, and $M_z$. This completes the proof.
\end{proof}

In the following corollary we express the above theorem in terms of pure $\Gamma_n$-isometries.

\begin{cor}
 Let $(S_1,\ldots,S_n)$ and $(\tilde S_1,\ldots, \tilde S_n)$ be a pair of pure $\Gamma_n$-isometries. Then $(S_1,\ldots,S_n)$ and $(\tilde S_1,\ldots, \tilde S_n)$ are unitarily equivalent if and only if the $(n-1)$-tuples $(S_{n-i}^* - S_iS_n^*)_{i=1}^{n-1}$ and $(\tilde S_{n-i}^* - \tilde S_i\tilde S_n^*)_{i=1}^{n-1}$  are unitary equivalent.
 \end{cor}

 \begin{proof}
  Let the pure $\Gamma_n$-isometry $(S_1,\ldots,S_n)$ is equivalent to $(M_{{\mathbf \Phi}},M_z)$ where $\Phi_i(z) = A_i + A_{n-i}^*z$ be in  $H^\infty \mathcal L(\mathcal E)$ for some  $A_i\in\mathcal L(\mathcal E),\,i=1,\ldots,n-1$ satisfying conditions (ii) to (iv) in Theorem \ref{u}. Clearly the $(n-1)$-tuple $(S_{n-i}^* - S_iS_n^*)_{i=1}^{n-1}$ is equivalent to $(M_{\Phi_{n-i}}^* - M_{\Phi_i}M_{\bar z})_{i=1}^{n-1}$. Note that with the canonical identification we have
  \beqa
  M_{\Phi_{n-i}}^* - M_{\Phi_i}M_{\bar z} &=& (I_{H^2}\otimes A_{n-i} + M_z\otimes A_i^*)^* - 
(I_{H^2}\otimes A_i + M_z\otimes A_{n-i}^*)(M_{\bar z}\otimes I_{\mathcal E})\\
  &=& (I_{H^2} - M_zM_{\bar z})\otimes A_{n-i}^*\\
  &=& P_{\C}\otimes A_{n-i}^*
  \eeqa
  where $P_{\C}$ is the orthogonal projection from $H^2$ to the scalars in $H^2$. Thus it follows for Theorem \ref{equi} that unitary equivalence of the $n$-tuple $(M_{{\mathbf \Phi}},M_z)$ is determined by unitary equivalence of the $(n-1)$-tuple $(M_{\Phi_{n-i}}^* - M_{\Phi_i}M_{\bar z})_{i=1}^{n-1}$,  the unitary equivalence of $(S_1,\ldots,S_n)$ is determined by $(S_{n-i}^* - S_iS_n^*)_{i=1}^{n-1}$. This completes the proof.
 \end{proof}

Characterization of invariant subspaces of $\Gamma_n$-isometries essentially boils down to that of pure $\Gamma_n$-isometries due to Wold type decomposition in Theorem \ref{wold} which is again same as characterizing invariant subspace for the associated model space obtained in Theorem \ref{p.i}. The following theorem discusses this issue.

A closed subspace $\mathcal M\neq \{0\}$ of $H^2(\m E)$ is said to be $(M_\mathbf \Phi,M_z)$-\emph{invariant} if $\m M$ is invariant
under $M_{\Phi_i}$ and $M_z$ for all $i=1,\ldots,n-1.$

Let $\m M\neq \{0\}$ be a closed subspace of $H^2(\m E_*).$ It follows from Beurling-Lax-Halmos theorem that $\m M$ 
is invariant under $M_z$ if and only 
if there exists a Hilbert space  $\m E$ and an inner function $\Theta\in H^\infty\m L(\m E,\m E_*)$ ($\Theta $ is an isometry
almost everywhere on $\mb T$) such that 
$$\m M=M_\Theta H^2(\m E).$$

 \begin{thm}
  Let $\mathcal M\neq 0$ be a closed subspace of $H^2({\mathcal E_*})$ and $\Phi_i, i=1,\ldots,n-1$ be in
$H^\infty{\mathcal L(\mathcal E_*)}$ such that $(M_{{\mathbf \Phi}},M_z)$ is a pure $\Gamma_n$-isometry
on $H^2(\mathcal E_*)$. Then $\mathcal M$ is a $(M_{{\mathbf \Phi}},M_z)$-invariant subspace if and only
if there exist $\Psi_i,\, i=1,\ldots,n-1$ in $H^\infty{\mathcal L(\mathcal E)}$ such that
$(M_{{\mathbf \Psi}},M_z)$ is a pure $\Gamma_n$-isometry on $H^2(\mathcal E)$ and
  $$
  {\Phi_i}{\Theta} = {\Theta}{\Psi_i}, \, i=1,\ldots,n-1,
  $$ where $\Theta \in H^\infty{\mathcal L(\mathcal E, \mathcal E_*)}$ is the Beurling-Lax-Halmos representation of $\mathcal M$.
 \end{thm}
\begin{proof}
 We will prove only the forward direction as the other part is easy to see. Let $\mathcal M$ is invariant under $(M_{{\mathbf \Phi}},M_z)$. Thus, in particular, $\mathcal M$ is invariant under $M_z$ and hence the Beurling -Lax-Halmos represntation of $\mathcal M$ is $\Theta H^2(\mathcal E)$ where $\Theta H^\infty{\mathcal L(\mathcal E, \mathcal E_*)}$ is an inner multiplier. We also have $M_{\Phi_i}\Theta H^2(\mathcal E)\subseteq \Theta H^2(\mathcal E)$ for each $i=1,\ldots,n-1$. Thus, there exist $\Psi_i$'s, $i=1,\ldots,n-1$ in $H^\infty{\mathcal L(\mathcal E)}$ such that ${M_{\Phi_i}}{M_\Theta} = {M_\Theta}{M_{\Psi_i}},$ that is, ${\Phi_i}{\Theta} = {\Theta}{\Psi_i},\,i=1,\ldots,n-1$. Since $\Phi_i$'s commute, we have ${M_\Theta}{M_{\Psi_i}}{M_{\Psi_j}} = {M_{\Phi_i}}{M_{\Phi_j}}{M_\Theta} = {M_{\Phi_j}}{M_{\Phi_i}}{M_\Theta} = {M_\Theta}{M_{\Psi_j}}{M_{\Psi_i}}$ and hence $\Psi_i\Psi_j = \Psi_j\Psi_i,\,i,j=1,\ldots,n-1$. Furthermore, we have
$$
p(M_{\Phi_1},\ldots,M_{\Phi_{n-1}})M_{\Theta} =  M_{\Theta}p(M_{\Psi_1},\ldots,M_{\Psi_{n-1}})
$$ for all polynomials $p\in \mb C[z_1,\ldots,z_{n-1}]$. Therefore,
$$
\|p(\g_1M_{\Psi_1},\ldots,\g_{n-1}M_{\Psi_{n-1}})\|\leq \|M_\Theta^*p(\g_1M_{\Phi_1},\ldots,\g_{n-1}M_{\Phi_{n-1}})M_{\Theta}\| \leq \Vert p\Vert_{\infty,\Gamma_{n-1},}
$$ for all polynomials $p\in \mb C[z_1,\ldots,z_{n-1}]$ and $\g_i=\frac{n-i}{n}$ for $i=1,\ldots,n-1$. Using Theorem \ref{p.i}, we also note that
$$
M_{\bar z} M_{\Psi_i} = M_{\bar z}M_{\Theta}^*M_{\Phi_i}M_{\Theta} = M_{\Theta}^*M_{\bar z}M_{\Phi_i}M_{\Theta} = M_{\Theta}^*M_{\Phi_{n-i}}^*M_{\Theta} = M_{\Psi_{n-i}}^*,\, i=1,\ldots,n-1.
$$ From the observations made above and by Theorem \ref{p.i}, it follows that $(M_{{\mathbf \Psi}},M_z)$ is a pure $\Gamma_n$-isometry on $H^2(\mathcal E)$ and this completes the proof.
\end{proof}

\end{document}